\documentclass[12pt,twoside]{article}

\usepackage{amsmath,amssymb,amsthm} 
\usepackage{supertabular}

\newcommand{\p}{\mathfrak{p}}

\newcommand{\Prim}{\mathrm{Prim}}

\newcommand{\Conj}{\mathrm{Conj}}

\newcommand{\vol}{\mathrm{vol}}

\newcommand{\tr}{\mathrm{tr}}
\newcommand{\li}{\mathrm{li}}

\newcommand{\Ker}{\mathrm{Ker}}

\newcommand{\lcm}{\mathrm{lcm}}

\newcommand{\mr}{\mathrm}
\newcommand{\Nsf}{N_{\mathrm{sf}}}

\newcommand{\as}{\quad\text{as}\quad}
\newcommand{\tinf}{\to\infty}
\newcommand{\disp}{\displaystyle}
\newcommand{\bsla}{\backslash}

\newcommand{\D}{\mathfrak{D}}

\newcommand{\cC}{\mathcal{C}}

\newcommand{\cT}{\mathcal{T}}
\newcommand{\cA}{\mathcal{A}}

\newcommand{\bC}{\mathbb{C}}
\newcommand{\bR}{\mathbb{R}}

\newcommand{\bZ}{\mathbb{Z}}

\newcommand{\noi}{\noindent}

\newcommand{\divset}{\hspace{3pt}|\hspace{3pt}}
\newcommand{\bigdivset}{\hspace{3pt}\big|\hspace{3pt}}
\newcommand{\Bigdivset}{\hspace{3pt}\Big|\hspace{3pt}}

\newcommand{\mmid}{\hspace{2pt}||\hspace{2pt}}

\newcommand{\gam}{\gamma}
\newcommand{\Gam}{\Gamma}
\newcommand{\hGam}{\hat{\Gamma}}

\newcommand{\slt}{\mathrm{SL}_2}
\newcommand{\psl}{\mathrm{PSL}_2}
\newcommand{\sr}{\mathrm{SL}_2(\bR)}
\newcommand{\sz}{\mathrm{SL}_2(\bZ)}

\newcommand{\vbpt}{\vspace{6pt}}

\newtheorem{thm}{Theorem}[section]
\newtheorem{prop}[thm]{Proposition}
\newtheorem{lem}[thm]{Lemma}
\newtheorem{rem}[thm]{Remark}
\newtheorem{cor}[thm]{Corollary}

\newtheorem{fact}{Fact}[section]

\numberwithin{equation}{section}

\title{Asymptotic formulas for class number sums of indefinite binary quadratic forms 
in arithmetic progressions}
\author{Yasufumi Hashimoto
\thanks{Partially support by JSPS Grant-in-Aid for Young Scientists (B) no. 20740027.}}
\date{}

\begin{document}
\markboth
{Y. Hashimoto}
{Class number sum}
\pagestyle{myheadings}

\maketitle
\renewcommand{\thefootnote}{}
\footnote{MSC2000: primary: 11E41; secondary: 11M36}

\begin{abstract}
It is known that there is a one-to-one correspondence between equivalence classes of
primitive indefinite binary quadratic forms and primitive hyperbolic conjugacy classes
of the modular group. Due to such a correspondence, Sarnak obtained the asymptotic
formula for the class number sum in order of the fundamental unit by using the prime
geodesic theorem for the modular group. In the present paper, we propose asymptotic
formulas of the class number sums over discriminants in arithmetic progressions. Since
there are relations between the arithmetic properties of the discriminants and the
conjugacy classes in the finite groups given by the modular group and its congruence
subgroups, we can get the desired asymptotic formulas by arranging the Tchebotarev-type
prime geodesic theorem. While such asymptotic formulas were already given by
Raulf, the approaches are quite different, the expressions of the leading terms of our
asymptotic formulas are simpler and the estimates of the reminder terms are sharper.
 \end{abstract}

\section{Introduction}

For an integer $D$, let $h(D)$ be the class number of $D$ in the narrow sense. 
In Sects. 302--304 of \cite{G}, Gauss stated the mean value formula for $h(D)$ without proof.
His formula for $D<0$ was proven by Lipschitz \cite{Li} and Mertens \cite{Me} 
and was improved by Vinogradov \cite{Vi}. 
For $D>0$, Siegel \cite{Si} proved that 
\begin{align}\label{siegel}
\sum_{\begin{subarray}{c}0<D<x\end{subarray}}h(D)\log{\epsilon(D)}
\sim \frac{\pi^2}{18\zeta(3)}x^{3/2}\as x\tinf,
\end{align}
where $\epsilon(D)$ is the fundamental unit of $D$ in the narrow sense and $\zeta(3):=\sum_{n\geq 1}n^{-3}$, 
and Shintani \cite{Sh} improved it. 
Such an asymptotic formulas has been further improved and extended in several ways 
by the theory of prehomogeneous vector spaces (see \cite{Sh}, \cite{Da}, \cite{GH} etc). 

On the other hand, Sarnak \cite{Sa1} obtained the following asymptotic formula.  
\begin{align}\label{sarnak1}
\sum_{\begin{subarray}{c} D>0,\epsilon(D)<x\end{subarray}}h(D)\log{\epsilon(D)}
\sim\frac{1}{2}x^2 \as x\tinf.
\end{align}
This yields that 
\begin{align}\label{sarnak2}
\sum_{\begin{subarray}{c} D>0,\epsilon(D)<x\end{subarray}}h(D)
\sim\li(x^2) \as x\tinf,
\end{align}
where $\li(x):=\int_2^x(\log{t})^{-1}dt$.  
The asymptotic formulas \eqref{sarnak1} and \eqref{sarnak2} follows from the prime geodesic theorem 
(see, e.g. \cite{Se} and \cite{He}) for $\sz$; 
\begin{align*}
\#&\{\text{$[\gam]$: the primitive hyperbolic conjugacy classes of $\sz$,} \\
&\text{the larger eigenvalue of $\gam$ is less than $x$}\}\sim \li(x^2) \as x\tinf
\end{align*} 
and the one-to-one correspondence between the equivalence classes 
of the primitive indefinite binary quadratic forms and the primitive conjugacy classes of $\sz$. 
Such asymptotic formulas have been extended to 
the binary quadratic forms over imaginary quadratic fields \cite{Sa2} 
and the ternary quadratic forms \cite{DH} by the theory of trace formulas. 

In the present paper, we study the growth of Sarnak-type class number sums over $D\equiv a \bmod{n}$ for given $a$ and $n$. 
The approach is as follows; (i) describe the relation between the
arithmetic property of the discriminants and the conjugacy classes in $\rm{PSL2}(\bZ/n\bZ)$, 
(ii) write down the Tchebotarev-type prime geodesic theorems (\cite{Sa1} and \cite{Su}) as asymptotic formulas
for partial class number sums and (iii) arrange such asymptotic formulas. The main result,
Theorem \ref{thm1}, gives the detail expressions of the coefficients of the leading terms and (not
necessarily best-possible but) non-trivial estimates of the reminder terms.

For such asymptotic formulas, Raulf \cite{Ra} already studied. The approach was quite
different, by reducing the problem to the estimation of the sums of special values of the
Dirichlet's $L$ functions with the class number formula. While she also gave expressions of the
coefficients of the leading terms, they were too complicated to be evaluated and
the estimates of the reminder terms are rough. Compared to her results, our leading terms are
simpler and the estimates of the reminder terms are sharper. In fact, the approximations
of the coefficients are smoothly computable as described in Example of Section 4.2. The
further advantage is that our approach will be arranged and extended easily. For example in
Theorem \ref{thm2}, discussing only the arithmetic properties of the discriminants, we obtain the
asymptotic formulas of the class number sums over fundamental discriminants in arithmetic
progressions. Moreover, after overcoming some problems for the prime geodesic theorems on
hyperbolic three manifolds, the results in this paper will be extended to the class number
sums of the binary quadratic forms over imaginary quadratic fields.

\section{Prime geodesic theorem} 

Let $H:=\{x+y\sqrt{-1}\in\bC\divset x,y\in\bR, y>0\}$ be the upper half plane with hyperbolic metric 
and $\Gam$ a discrete subgroup of $\sr$ with $\vol(\Gam\bsla H)<\infty$.
Denote by $\Prim(\Gam)$ the set of primitive hyperbolic conjugacy classes of $\Gam$
and $N(\gam)$ the square of the larger eigenvalue of $\gam\in\Prim(\Gam)$. 
For a finite dimensional unitary representation $\chi$, 
it is known that 
\begin{align}\label{pgt}
\pi_{\Gam,\chi}(x):=&\sum_{\begin{subarray}{c}\gam\in\Prim(\Gam)\\ N(\gam)<x\end{subarray}}\tr\chi(\gam)
=\sum_{0\leq \lambda_{j,\chi}\leq 1/4}\li\left(x^{s_{j,\chi}}\right)
+R_{\Gam,\chi}(x),
\end{align}
where $\lambda_{j,\chi}$ is the $j$-th eigenvalue of the Laplacian acting on the
sections of the flat vector bundle on $\Gam\bsla H$ associated to the 
representation $\chi$, $s_{j,\chi}:=1/2+\sqrt{1/4-\lambda_{j,\chi}}$ 
and $R_{\Gam,\chi}(x)$ is the essential reminder term for 
the asymptotic behavior of $\pi_{\Gam,\chi}(x)$. 
Note that \eqref{pgt} is called by the prime geodesic theorem for $(\Gam,\chi)$ 
and $R_{\Gam,\chi}(x)$ is presently bounded by $R_{\Gam,\chi}(x)=O(x^{3/4})$ 
(see, e.g. \cite{Se} and \cite{He}). 
For the implied constant of its estimate of $R_{\Gam,\chi}(x)$, 
Jorgenson and Kramer \cite{JK} proved the following lemma. 
\begin{lem}(\cite{JK} and also Lemma 9.6.2 in \cite{Bu})\label{jk}
Let $\Gam$ be a discrete subgroup of $\sr$ with $\vol(\Gam\bsla H)<\infty$ 
and $\chi$ a finite dimensional unitary representation of $\Gam$. 
Then, there exists a constant $A_{\Gam}>0$ depending on $\Gam$ such that 
\begin{align}
\left|R_{\Gam,\chi}(x)\right|\leq \dim\chi A_{\Gam}x^{3/4}
\end{align} 
for sufficiently large $x>0$.
\end{lem}

Let $\Gam'$ be a normal subgroup of $\Gam$ with $[\Gam:\Gam']<\infty$. 
Denote by $G:=\Gam/\Gam'$ and $\iota:\Gam\to G$ the natural projection. 
According to \cite{Sa1} and \cite{Su}, 
we have the following asymptotic formula for a conjugacy class $[g]$ in $G$.
\begin{align}\label{chebotarev}
\pi_{\Gam}(x;\Gam',[g]):=&\#\{\gam\in\Prim(\Gam)\divset \iota(\gam)\subset [g],N(\gam)<x\}
\sim \frac{\#[g]}{\# G}\li(x).
\end{align}
This can be interpreted as an analogue of the Tchebotarev density theorem 
for algebraic number fields (see, e.g. \cite{Ar}, \cite{Ta} and \cite{Tc}). 
Put 
\begin{align*}
R_{\Gam}(x;\Gam',[g]):=\pi_{\Gam}(x;\Gam',[g])-\frac{\#[g]}{\# G}\li(x).
\end{align*}
We now estimate $R_{\Gam}(x;\Gam',[g])$ as follows. 

\begin{lem}\label{errorchebo}
Let $\Gam$ be a discrete subgroup of $\sr$ with $\vol(\Gam\bsla H)<\infty$ 
and $\Gam'$ a normal subgroup of $\Gam$ with $[\Gam:\Gam']<\infty$. 
Then there exists a constant $B_{\Gam}>0$ depending on $\Gam$ such that 
\begin{align*}
\left|R_{\Gam}(x;\Gam',[g])\right|\leq 
\#[g]B_{\Gam}x^{c_{\Gam,\Gam'}}
\end{align*}
for sufficiently large $x>0$, where 
\begin{align*}
\disp c_{\Gam,\Gam'}:=\max\left\{3/4,s_{j,\chi}\divset \chi \in \hat{G},0<\lambda_{j,\chi}\leq 1/4\right\}.
\end{align*}
\end{lem}

\begin{proof}
For an element $g\in G$, it holds that 
\begin{align}
\sum_{\chi\in\hat{G}}\tr\chi(g^{-1})\pi_{\Gam,\chi}(x)=\frac{\#G}{\#[g]}\pi_{\Gam}(x;\Gam',[g]),
\end{align}
(see \cite{Su}). 
By virtue of \eqref{pgt}, we have
\begin{align*}
\pi_{\Gam}(x;\Gam',[g])=\frac{\#[g]}{\#G}\sum_{\chi\in\hat{G}}\tr\chi(g^{-1})
\left(\sum_{0\leq \lambda_{j,\chi}\leq 1/4}\li\left(x^{s_{j,\chi}}\right)
+R_{\Gam,\chi}(x)\right).
\end{align*}
Put 
\begin{align*}
R^{(1)}(x):=&\frac{\#[g]}{\#G}\sum_{\chi\in\hat{G}}\tr\chi(g^{-1})
\sum_{0< \lambda_{j,\chi}\leq 1/4}\li\left(x^{s_{j,\chi}}\right),\\
R^{(2)}(x,T):=&\frac{\#[g]}{\#G}\sum_{\chi\in\hat{G}}\tr\chi(g^{-1})R_{\Gam,\chi}(x).
\end{align*}
Due to Lemma \ref{jk}, we get 
\begin{align*}
\left|R^{(2)}(x,T)\right|
\leq & \frac{\#[g]}{\#G}A_{\Gam}x^{3/4}\sum_{\chi\in\hat{G}}\left|\tr\chi(g^{-1})\right|\dim\chi
\leq \#[g]A_{\Gam}x^{3/4}.
\end{align*}
Since $\#\left\{ 0< \lambda_{j,\chi}\leq 1/4\right\}\leq B'_{\Gam}\dim{\chi}$ 
for some constant $B'_{\Gam}>0$ depending on $\Gam$
(\cite{Bur}, \cite{Zo} and \cite{JK}), we also have
\begin{align*}
\left|R^{(1)}(x)\right|\leq &\frac{\#[g]}{\#G}\sum_{\chi\in\hat{G}}\left|\tr\chi(g^{-1})\right|
\sum_{0< \lambda_{j,\chi}\leq 1/4}\li\left(x^{\max(s_{j,\chi})}\right)
\leq \#[g] B'_{\Gam}\li\left(x^{\max(s_{j,\chi})}\right).
\end{align*}
Thus we obtain the estimate in the lemma.
\end{proof}

In Section 4, we will use Lemma \ref{errorchebo} 
in the case that $\Gam=\sz$ and $\Gam'$ is a principal congruence subgroup 
to get the main results. 
For such $\Gam$ and $\Gam'$, it is known that 
$\lambda_{j,\chi}=0$ or $\lambda_{j,\chi}\geq975/4096$ 
(see \cite{Se}, \cite{LRS}, \cite{KS} and \cite{Ki}). 
Then we can take $c_{\sz,\hGam(n)}=3/4$. 

Remark that it has been proved that $R_{\Gam}(x)\ll x^{7/10}$ 
when $\Gam$ is a congruence subgroup of $\sz$ (see \cite{LS}, \cite{LRS} and also \cite{Ko}). 
While it is better than $R_{\Gam}(x)\ll x^{3/4}$, 
we require the estimate with implied constants as described in Lemma \ref{jk} and \ref{errorchebo}. 
Of course, getting the estimate $R_{\Gam}(x)\ll x^{7/10}$ with implied constants 
is not necessarily impossible, and it will give better results if possible. 
However, discussing it is exhaustive and will make this paper too heavy. 
Thus we will apply Lemma \ref{errorchebo} in this paper.

\section{Modular group and binary quadratic forms}

\subsection{Relations between the modular group and quadratic forms}
Let 
\begin{align*}
Q(x,y)=[a,b,c]:=ax^2+bxy+cy^2
\end{align*}
be a binary quadratic form over $\bZ$ with $a,b,c\in\bZ$ and $\gcd(a,b,c)=1$. 
We call that quadratic forms $Q$ and $Q'$ are equivalent 
and write $Q\sim Q'$ if there exists $g\in\sz$ such that $Q(x,y)=Q'\big((x,y).g\big)$. 
Denote by $h(D)$ the number of equivalence classes of the quadratic forms of given $D=b^2-4ac$. 
Let $D=D(Q):=b^2-4ac$ be the discriminant of $[a,b,c]$. 
It is known that, if $D>0$, then there are infinitely many positive solutions $(t,u)$ 
of the Pell equation $t^2-Du^2=4$. 
Put $(t_j,u_j)=(t_j(D),u_j(D))$ the $j$-th positive solution of $t^2-Du^2=4$ 
and $\epsilon_j(D):=(t_j(D)+u_j(D))/2$. 
Note that $\epsilon(D)=\epsilon_1(D)$ is called by the fundamental unit of $D$ in the narrow sense, 
and it holds that $\epsilon_j(D)=\epsilon(D)^j$.

For a quadratic form $Q=[a,b,c]$ and a solution $(t,u)$ of $t^2-Du^2=4$, we put 
\begin{align}\label{1to1}
\gam\big(Q,(t,u)\big):=\begin{pmatrix}\disp\frac{t+bu}{2}&-cu\\ au& \disp\frac{t-bu}{2}\end{pmatrix}\in\slt(\bZ).
\end{align}
Conversely, for $\gam=(\gam_{i,j})_{1\leq i,j\leq 2}\in \sz$, we put
\begin{align}
&t_{\gam}:=\gam_{11}+\gam_{22},\quad u_{\gam}:=\gcd{(\gam_{21},\gam_{11}-\gam_{22},-\gam_{12})},\notag\\
&a_{\gam}:=\gam_{21}/u_{\gam},\quad b_{\gam}:=(\gam_{11}-\gam_{22})/u_{\gam},
\quad c_{\gam}:=-\gam_{12}/u_{\gam}, \label{ttod}\\
&Q_{\gam}:=[a_{\gam},b_{\gam},c_{\gam}],\quad 
D_{\gam}:=\frac{t_{\gam}^2-4}{u_{\gam}^2}=b_{\gam}^2-4a_{\gam}c_{\gam}.\notag
\end{align}
It is known that \eqref{1to1} and \eqref{ttod} gives a one-to-one correspondence between
equivalence classes of primitive binary quadratic forms $D>0$ 
and the elements of $\Prim(\sz)$. 
We now note the following elementary facts without proof (see, e.g. \cite{G}).
\begin{fact}\label{factquad}
Suppose that $\gam,\gam_1,\gam_2\in\sz$ are not in any finite group in $\sz$. Then we have \\
1. $\gam(Q_{\gam},(t_{\gam},u_{\gam}))=\gam$.\\
2. $D(Q)=D_{\gam(Q,(t,u))}$ for any $t,u\geq1$ with $t^2-Du^2=4$, and $D_{\gam}=D(Q_{\gam})$.\\
3. $(t_{\gam(Q,(t,u))},u_{\gam(Q,(t,u))})=(t,u)$, and 
$(t_{\gam},u_{\gam})$ coincides one of $(t_j(D),u_j(D))$.\\
4. If $Q_{\gam_1}=Q_{\gam_2}$ then $Q_{\gam_1\gam_2}=Q_{\gam_1}=Q_{\gam_2}$, 
and $(t_{\gam_1\gam_2},u_{\gam_1\gam_2})=
\left( \frac{1}{2}(t_{\gam_1}t_{\gam_2}+Du_{\gam_1}u_{\gam_2}\right)$, 
$\frac{1}{2}(t_{\gam_1}u_{\gam_2}+t_{\gam_2}u_{\gam_1})\big)$. \\
5. $Q_{g^{-1}\gam g}(x,y)=Q_{\gam}\big((x,y).g\big)$ for any $g\in\sz$.
\end{fact}
Since $N(\gam)$ for $\gam\in\Prim(\sz)$ coincides $\epsilon_1(D)^2$ of corresponding 
quadratic form, the prime geodesic theorem for $\sz$ yields the following asymptotic formula. 
\begin{align*}
\pi_{\sz}(x^2)=\sum_{\begin{subarray}{c}D\in\D\\ \epsilon(D)<x\end{subarray}}h(D)
=\li(x^2)+R_{\sz}(x^2),
\end{align*}
where $\D:=\{D>0\divset \text{$D\equiv 0,1\bmod{4}$, not square}\}$.

\subsection{On $\mathrm{PSL}_2(\bZ/n\bZ)$} 

First we prepare some notations. 

For integers $n,m\geq1$, $m\mmid n$ means that 
$m\mid n$ and $\gcd(m,n/m)=1$, 
and $n_m$ denotes the maximal divisor of $n$ relatively prime to $m$. 
For an integer $n\geq1$, let 
\begin{align*}
\bZ_{n}:=&\bZ/n\bZ,\quad \quad 
\bZ_{n}^{(2)}:=\bZ_{n}^*/\left(\bZ_{n}^*\right)^2,\\
\hat{\Gam}(n):=&\Ker\big(\mr{SL}_2(\bZ)\stackrel{\text{proj.}}{\rightarrow}
\mr{PSL}_2(\bZ_n)\big)\\
=&\{\gam\in \mr{SL}_2(\bZ)\divset \gam\equiv \alpha I\bmod{n}, \alpha^2\equiv 1\bmod{n} \}.
\end{align*}
The subgroup $\hGam(n)$ of $\sz$ is called by the principle congruence subgroup of level $n$.
If the factorization of $n$ is $n=\prod_{p}p^r$ then 
\begin{align*}
&\hGam(n)=\bigcap_{p\mid n}\hGam(p^r), \\ 
&v(n):=[\slt(\bZ):\hGam(n)]=\#\mr{PSL}_2(\bZ_n)
=\prod_{p\mid n}\frac{p^{3r-2}(p^2-1)}{\#\bZ_{p^r}^{(2)}}.
\end{align*} 
Note that (see, e.g. \cite{G}) 
\begin{align*}
\#\bZ_{p^r}^{(2)}=
\begin{cases}1, &(p^r=2),\\ 2, &(\text{$p^r=4$ or $2\nmid p$}),\\
4,& (\text{$p=2,r\geq3$}).\end{cases}
\end{align*}

We now prepare the following lemma. 
\begin{lem}\label{cong}
Let $n\geq1$ be an integer and $\gam\in\sz$. 
Then $\gam\in\hGam(n)$ if and only if $n\mid u_{\gam}$.
\end{lem}
The proof for prime $n$ was already done in \cite{Sa1}, 
and its generalization to composite $n$ is similar. \qed

Let $p$ be a prime and $k,r\geq1$ integers. 
Due to \eqref{1to1} and Lemma \ref{cong}, 
we can express the element $\gam\in\hGam(p^{k})-\hGam(p^{k+1})$ in $\psl(\bZ_{p^{r+k}})$ by 
\begin{align}\label{modp}
\gam\equiv \begin{pmatrix}\disp\frac{T+Bp^k}{2} & Ap^k \\ -Cp^k & \disp\frac{T-Bp^k}{2}\end{pmatrix}
\bmod{p^{r+k}},
\end{align}
where $A\equiv a_{\gam}(u_{\gam}/p^k)$, $B\equiv b_{\gam}(u_{\gam}/p^k)$, 
$C\equiv c_{\gam}(u_{\gam}/p^k) \bmod{p^r}$ and $T\equiv t_{\gam}\bmod{p^{r+k}}$. 
Note that $p^k\mmid u_{\gam}$ and $T$ satisfies 
\begin{align}
T^2\equiv 4+(B^2-4AC)p^{2k}
\bmod{\begin{cases}2^{r+k+2}&(p=2),\\ p^{r+k}&(p\geq3).\end{cases}}\label{tsquare}
\end{align}
Let $u_{n}:=(u_{\gam})_n$ and $D(n)=(D_{\gam})(n):=u_{n}^2D_{\gam}$. 
We see that $D(p)=B^2-4AC$. 

We now prove the following proposition as a preparation for the main theorem.

\begin{prop}\label{count}
Let $r,k$ be integers, $p$ a prime number and $\delta\in \bZ_{p^r}$. 
Put 
\begin{align*}
&\cT(\delta;p^r,k):=\#\left\{T\in\bZ_{p^{r+k}}\Bigdivset T^2\equiv 4+\delta p^{2k}
\bmod{\begin{cases}2^{r+k+2}&(p=2)\\ p^{r+k}&(p\geq3)\end{cases}}\right\},\\
&\cA(\delta;p^r):=\#\left\{(A,B,C)\in\bZ_{p^{r}}^3\bigdivset p\nmid \gcd(A,B,C), 
B^2-4AC\equiv \delta\bmod{p^r}\right\},\\
&\hGam(\delta;p^r,k):=
\left\{\gam\in\psl(\bZ_{p^{r+k}}) \Bigdivset D(p)\equiv \delta\bmod p^r\right\}.
\end{align*}
Then it holds that 
\begin{align*}
\#\hGam(\delta;p^r,k)=\cT(\delta;p^r,k) \cA(\delta;p^r)/\#\bZ_{p^r}^{(2)}.
\end{align*}
Furthermore, $\cT(\delta;p^r,k)$ and $\cA(\delta;p^r)$ are given as follows.
\begin{align*}
\cT(\delta;2^r,k)=&\begin{cases}4,&(32\mid 2^{2k}\delta),\\
2^{l+1},&(k=0,\delta\equiv-4+2^{2l}\alpha^2\bmod {2^{r+2}},0\leq l\leq r/2),\\
2^{[r/2]},&(k=0, \delta\equiv-4\bmod{2^{r+2}},\\
&\text{or $k=0, \delta\equiv-4+2^{r+1}\bmod{2^{r+2}}$, $r$ is odd}),\\
0,&(\text{otherwise}),\end{cases}\\
\cT(\delta;p^r,k)=&\begin{cases}2,&(p\mid \delta p^{2k}),\\
2p^l,&(k=0,\delta\equiv-4+\alpha^2 p^{2l}\bmod p^{r},0\leq l< r/2),\\
p^{[r/2]},&(k=0,\delta\equiv -4\bmod{p^r}),\\
0,&(\text{otherwise}),\end{cases},\\
\cA(\delta;2^{r+2})=&\begin{cases}3\times 2^{2r-2},&\left(\delta\equiv 1\bmod{8}\right),\\
2^{2r-2},&\left(\delta\equiv 5\bmod{8}\right),\\
3\times 2^{2r-3},&(4 \mid \delta),\\
0,&(\text{otherwise}),\end{cases}\\
\cA(\delta;p^r)=&\begin{cases}p^{2r-1}(p+1),&\left((\delta/p)=1\right),\\
p^{2r-1}(p-1),&\left((\delta/p)=-1\right),\\
p^{2r-2}(p^2-1),&(p \mid \delta),\end{cases}
\end{align*}
where $p$ is an odd prime and $\alpha\in \bZ_{2^{r+2}}^*$ or $\bZ_{p^r}^{*}$.
\end{prop}

\begin{proof}
We first note that $\cT(\delta;p^r,k)$ does not depend on $(A,B,C)$ but $\delta$. 
Then we have 
\begin{align*}
\#\hGam(\delta;p^r,k)=\cT(\delta;p^r,k) \cA(\delta;p^r)/\#\bZ_{p^r}^{(2)}.
\end{align*}

Next, compute $\cT(\delta;p^r,k)$. 
The equation is 
\begin{align*}
T^2\equiv 4+\delta p^{2k}\bmod{\begin{cases}2^{r+k+2}&(p=2)\\ p^{r+k}&(p\geq3)\end{cases}}.
\end{align*} 
It is easy to see that this has solutions if and only if $4+\delta p^{2k}\equiv p^{2l}\alpha^2$ or $\equiv 0$ 
for some $l\geq0$ and $\alpha\in\bZ_{p^r}^{*}$. 
The solutions of $T^2\equiv \alpha^2p^{2l}\bmod{p^r}$ are written in the form 
$T\equiv\pm\alpha p^{l}+\beta p^{r-l}$ for $\beta\in\bZ_{p^l}$ when $p\geq3$, 
and are $T\equiv\omega\alpha p^l+\beta 2^{r-l-1}$ for $\omega\in\bZ_{p^r}^{(2)}$ 
and $\beta\in\bZ_{p^{l+1}}$ when $p=2$. 
On the other hand, the solutions of $T^2\equiv 0\bmod{p^r}$ are in the form $T\equiv \beta p^l$ with $2l\geq r$. 
Thus we can get $\cT(\delta;p^r,k)$ as described in the proposition.

Finally, we compute $\cA(\delta;p^r)$. 

\noi{\bf The case of $\delta\equiv 5\bmod{8}$, $p=2$.} 
In this case, $B$ must be odd and the number of such $B$ is $2^{r-1}$.
Since $(\delta-B^2)/4$ is odd, $C$ is given by 
$C\equiv A^{-1}\frac{\delta-B^2}{4}$ for given $A\in\bZ_{2^r}^{*}$ and $B$. 
Thus we have $\cA(\delta;2^{r+2})=2^{r-1}\times 2^{r-1}=2^{2r-2}$.

\noi{\bf The case of $4\mid\delta$, $p=2$.}
When $4\mid\delta$, $B$ must be even. 
If $(\delta-B^2)/4\not\equiv 0\bmod{2}$, then $C$ is determined by 
$C\equiv A^{-1}\frac{\delta-B^2}{4}$ for given $A\in\bZ_{2^r}^{*}$ and $B$. 
Then the number of such $(A,B,C)$ is $2^{r-2}\times 2^{r-1}=2^{2r-3}$. 
If $(\delta-B^2)/4\equiv 0\bmod{2}$, then $2\mid AC$. 
Since $2\nmid \gcd(A,B,C)$, one of $A,C$ is odd. 
This means that the number of such $(A,B,C)$ is $2^{r-2}\times 2\times 2^{r-1}=2^{2r-2}$. 
Thus we have 
$\cA(\delta;2^{r+2})=2^{2r-3}+2^{2r-2}=3\times2^{2r-3}$.

\noi{\bf The case of $\delta\equiv 1\bmod{8}$, $p=2$.} 
In this case, $B$ must be odd and $AC$ even. 
Putting $B=1+2B_1$ and $\delta=1+8\delta_1$ into the equation $B^2-4AC\equiv \delta\bmod{2^{r+2}}$, 
we have
\begin{align*}
B_1^2+B_1\equiv \delta_1+\frac{AC}{2}\bmod{2^{r-1}}.
\end{align*} 
If the right hand side of the above is even, 
then there are two $B_1\in\bZ_{2^{r-1}}$, 
and if it is odd then there are no such $B_1$. 
It is easy to see that the number of $(A,C)$ with $2\mid AC$ 
and $2\mid \frac{AC}{2}-\delta_1$ is $3\times 2^{2r-3}$. 
Thus we have $\cA(\delta;2^{r+2})=3\times 2^{2r-3}\times 2=3\times 2^{2r-2}$.

\noi{\bf The case of $(\delta/p)=-1$, $p\geq3$.} 
When $(\delta/p)=-1$, we have $\delta-B^2=4AC\not\equiv0\bmod{p}$ for any $B\in\bZ_{p^r}$. 
Then $C$ is uniquely determined by $C\equiv(4A)^{-1}(\delta-B^2)$ for given $B\in\bZ_{p^r}$ 
and $A\in \bZ_{p^r}^{*}$. 
Thus we have $\cA(\delta;p^r)=p^r\times p^{r-1}(p-1)=p^{2r-1}(p-1)$.

\noi{\bf The case of $p\mid\delta$, $p\geq3$.}
When $p\nmid B$, $C$ is uniquely determined by $C\equiv (4A)^{-1}(\delta-B^2)\bmod{p^r}$ 
for given $A\in\bZ_{p^r}^{*}$. 
Then the number of such $(A,B,C)$ is $p^{r-1}(p-1)\times p^{r-1}(p-1)=p^{2r-2}(p-1)^2$.
When $p\mid B$, we have $p\mid AC$. 
Since $p\nmid\gcd(A,B,C)$, one of $A$, $C$ is not divided by $p$. 
Then the number of such $(A,B,C)$ is $p^{r-1}\times 2\times p^{r-1}(p-1)=2p^{2r-2}(p-1)$.
Thus we have $\cA(\delta;p^r)=p^{2r-2}(p-1)^2+2p^{2r-2}(p-1)=p^{2r-2}(p^2-1)$.

\noi{\bf The case of $(\delta/p)=1$, $p\geq3$.} 
When $(\delta/p)=1$, there are $B\in\bZ_{p^r}$ such that $p\mid \delta-B^2$. 
Since the number of $B\in\bZ_{p^r}$ with $B^2\equiv \delta\bmod{p^l}$ ($1\leq l\leq r$) 
is $2p^{r-l}$, 
the number of $B\in\bZ_{p^r}$ with $p^l\mmid \delta-B^2$ is $p^{r-1}(p-2)$ for $l=0$, 
$2p^{r-l-1}(p-1)$ for $1\leq l\leq r-1$ and $2$ for $l=r$. 

When $p\nmid \delta-B^2$, $C$ is uniquely determined by $C\equiv (4A)^{-1}(\delta-B^2)$ 
for given $A\in\bZ_{p^r}^{*}$. 
Then the number of such $(A,C)$ is $p^{r-1}(p-1)$. 
When $\delta-B^2\equiv \alpha p^l$ for $1\leq l\leq r-1$ and $\alpha\in\bZ_{p^{r-l}}^{*}$, 
$A,C$ are given by $A\equiv a_1p^{l_1}$ and $C\equiv c_1p^{l-l_1}$ for $0\leq l_1\leq l$, 
$a_1\in\bZ_{p^{r-l}}^{*}$ and $c_1\in\bZ_{p^{r+l_1-l}}^{*}$. 
Since $a_1c_1\equiv \alpha\bmod{p^{r-l}}$, 
we see that the number of such $(a_1,c_1)$ is $p^{r-1}(p-1)$ for given $l$ and $l_1$.
When $p^r\mid \delta-B^2$, $A$ and $C$ are given by 
$A\equiv a_1p^{l_1}$ and $C\equiv c_1p^{r-l_1}$ for $0\leq l_1\leq r$, 
$a_1\in\bZ_{p^{r-l_1}}^{*}$ and $c_1\in\bZ_{p^{l_1}}$. 
Then the number of such $(A,C)$ is 
\begin{align*}
\sum_{0\leq l_1\leq r-1}p^{r-l_1-1}(p-1)\times p^{l_1}+p^r=(r+1)p^r-rp^{r-1}.
\end{align*}
Thus we get 
\begin{align*}
\cA(\delta;p^r)=&p^{r-1}(p-2) p^{r-1}(p-1)
+\sum_{1\leq l\leq r-1}2p^{r-l-1}(p-1) (l+1) p^{r-1}(p-1)\\
&+2\left((r+1)p^r-rp^{r-1}\right)\\=&p^{2r-1}(p+1).
\end{align*}

\end{proof}

\section{Main results}

Let $\cC$ be a ``condition" for $D\in\D$ and 
$\D(\cC)$ the set of $D\in\D$ satisfying the condition $\cC$. 
For example, $\D(D\equiv 1\bmod{4})$ is the set of $D\in\D$ with $D\equiv 0,1\bmod{4}$. 
For a condition $\cC$, denote by 
\begin{align*}
\pi(x;\cC):=&\sum_{\begin{subarray}{c}D\in\D(\cC)\\ \epsilon(D)<x\end{subarray}}h(D).
\end{align*}
If $\pi(x,\cC)/\li(x^2)$ converges as $x\tinf$, put 
\begin{align*}
\eta(\cC):=&\lim_{x\tinf}\frac{\pi(x,\cC)}{\li(x^2)},\\ 
R(x;\cC):=&\pi(x,\cC)-\eta(\cC)\li(x^2).
\end{align*}
The aim in this section is to study $\pi(x;\cC)$ 
when $\cC$ is given by the arithmetic progressions. 

\subsection{Lemmas}

\begin{lem}\label{lem1}
Let $n_1,n_2\geq1$ be integers and $\delta\in\bZ_{n_1}$. 
Then $\eta(D_{n_1}\equiv\delta\bmod{n_1},n_2\mid u)$ exists and 
\begin{align*}
&\left|R(x;D(n_1)\equiv\delta\bmod{n_1},n_2\mid u)\right|\\
&\leq \eta(D(n_1)\equiv\delta\bmod{n_1},n_2\mid u)v(n_1n_2)B_{\sz}x^{3/2}.
\end{align*}
\end{lem}

\begin{proof} 
Due to Fact \ref{factquad}, we see that the element 
\begin{align*}
\gam\equiv\begin{pmatrix}\disp\frac{t+Bn_2}{2}&An_2\\-Cn_2&\disp\frac{t-Bn_2}{2}\end{pmatrix}
\bmod{n_1n_2}
\end{align*}
with $\delta\equiv B^2-4AC\bmod{n_1}$ and $n_1\nmid \gcd(A,B,C)$ corresponds 
to the quadratic form with $D(n_1)\equiv\delta\bmod{n_1}$ and $n_2 \mid u$. 
Then, according to Lemma \ref{errorchebo}, we have
\begin{align*}
&\pi(x;D(n_1)\equiv\delta\bmod{n_1}, n_2\mid u)\\
&=\sum_{\begin{subarray}{c}[g]\in\Conj(\sz)\\ (D_g)(n_1)\equiv\delta\bmod{n_1},
n_2\mid u_g\end{subarray}}\#[g]\pi_{\sz}\left(x;\hGam(n_1n_2),[g]\right).
\end{align*}
Thus we get
\begin{align*}
&\eta(D(n_1)\equiv\delta\bmod{n_1},n_2\mid u)=\sum_{[g]}\frac{\#[g]}{v(n_1n_2)}, \\
&\left|R(x;D(n_1)\equiv\delta\bmod{n_1},n_2\mid u)\right|=\left|\sum_{[g]}R_{\sz}\left(x^2;\hGam(n_1n_2),[g]\right)\right|
\leq \sum_{[g]}\#[g]B_{\sz}x^{3/2},
\end{align*}
and the lemma follows immediately.
\end{proof}

When $n_1,n_2$ are powers of the same prime $p$, 
we can easily find the following lemma.

\begin{lem}\label{lempr}
Let $p$ be a prime number, $k\geq0,r\geq1$ integers and $\delta \in \bZ_{p^r}$. 
Then we have
\begin{align*}
\eta\left(D(p)\equiv\delta\bmod{p^r}, p^k\mmid u\right)
=\frac{\#\hGam(\delta,p^r,k)}{v(p^{r+k})},
\end{align*}
where $\#\hGam(\delta,p^r,k)$ is explicitly calculated in Proposition \ref{count}.
\end{lem} 

For composite $n_1$, we give the following lemmas.
\begin{lem}\label{modpq}
Let $n_1,n_2\geq1$ be integers relatively prime to each other 
and fix $[g]$ a conjugacy class in $\psl(\bZ_{n_1n_2})$. 
Then we have
\begin{align*}
\#[g]=&\#[g]_{n_1} \#[g]_{n_2},\\
\end{align*}
where $[g]_{n_1},[g]_{n_2}$ are the sets of $g_1\in\psl(\bZ_{n_1})$ 
with $g_1\equiv g\bmod{n_1}$ and 
$g_2\in\psl(\bZ_{n_2})$ with $g_2\equiv g\bmod{n_2}$ for some $g\in[g]$ respectively.
\end{lem}
\begin{proof}
See the discussions in the proof of Proposition 4.5 in \cite{HW}. 
\end{proof}

\begin{lem}\label{factor}
Let $n_1,n_2\geq1$ be integers 
and $n_1=p_1^{r_1}\cdots p_l^{r_l}$ the factorization of $n_1$. 
Denote by $m_1$ the maximal divisor of $n_2$ relatively prime to $n_1$, 
namely $n_2=p_1^{k_1}\cdots p_l^{k_l}m_1$ for some $k_1,\cdots,k_l\geq0$
and relatively prime $m_1,n_1\geq1$. 
Then we have
\begin{align*}
&\eta\left(D(n_1)\equiv \delta\bmod{n_1},n_2\mid u\right)\\
&=\eta(m_1\mid u)
\prod_{i=1}^l\eta\left(D(p_i)\equiv (n_1)_{p_i}^2\delta\bmod{p^{r_i}},p_i^{k_i}\mid u\right).
\end{align*}
\end{lem}

\begin{proof}
We first note that the condition ``$D(n_1)\equiv \delta\bmod{n_1},n_2\mid u$" 
is equivalent that ``$D(p_i)\equiv (n_1)_{p_i}^2\delta\bmod{p^{r_i}},p_i^{k_i}\mid u$ 
for any $1\leq i\leq l$ and $m_1\mid u$". 
As discussed in the proofs of Lemma \ref{lem1} and \ref{cong}, 
it holds that
\begin{align*}
\eta\left(D(n_1)\equiv \delta\bmod{n_1},n_2\mid u\right)
=&\sum_{\begin{subarray}{c}[g]\in\Conj(\psl(\bZ_{n_1n_2}))\\ 
(D_g)(n_1)\equiv \delta\bmod{n_1},n_2\mid u_g \end{subarray}}\frac{\#[g]}{v(n_1n_2)},\\
\eta\left(D(p_i)\equiv (n_1)_{p_i}^2\delta\bmod{p^{r_i}},p_i^{k_i}\mid u\right)
=&\sum_{\begin{subarray}{c}[g]\in\Conj(\psl(\bZ_{n_1n_2}))\\ 
(D_g)(p_i)\equiv (n_1)_{p_i}^2\delta\bmod{p_i^{r_i}}\\p_i^{k_i}\mid u_g \end{subarray}}
\frac{\#[g]}{v(p_i^{r_i+k_i})},\\
\eta(m_1\mid u)=&\frac{1}{v(m_1)}.
\end{align*}
Then the desired result can follow easily from Lemma \ref{modpq}.
\end{proof}

We now prove the following proposition. 
This is important to prove the main theorems.
\begin{prop}\label{proppr}
Let $p$ be a prime number, $r\geq1$ an integer and $\delta\in\bZ_{p^r}$. 
Then $\eta(D(p)\equiv\delta\bmod{p^r})$ exists with
\begin{align*}
\eta(D(p)\equiv\delta\bmod{p^r})
=\sum_{k=0}^{\infty}\eta(D(p)\equiv\delta\bmod{p^r},p^k\mmid u)
=\sum_{k=0}^{\infty}\frac{\#\hGam(\delta,p^r,k)}{v(p^{r+k})},
\end{align*}
and the reminder term is estimated by 
\begin{align*}
\left|R(x;D(p)\equiv\delta\bmod{p^r})\right|
\leq \sum_{k=0}^{\lfloor \log_p x\rfloor}\#\hGam(\delta,p^r,k)B_{\sz}x^{3/2}.
\end{align*}
\end{prop}

\begin{proof}
Since $u^2$ is a divisor of $t^2-4$, 
we have $u<t<\epsilon(D)+1$.  
Then we see that 
\begin{align*}
&\pi(x;D(p)\equiv\delta\bmod{p^r})=\sum_{k\geq0, p^k<x}\pi(x;D(p)\equiv\delta\bmod{p^r},p^k\mmid u).
\end{align*}
Divide the right hand side of the above as follows.
\begin{align*}
\sum_k\left( \eta(*)\li(x^2)+R(*) \right)=:L_1(x)\li(x^2)+L_2(x).
\end{align*}
Since $\eta(D(p)\equiv\delta\bmod{p^r},p^k\mmid u)\leq \eta(p^k\mid u)=v(p^k)^{-1}\leq Cp^{-3k}$, 
we have 
\begin{align*}
L_1(x)=\lim_{x\tinf}L_1(x)+O(x^{-3}).
\end{align*}
The estimation of $R(x;D(p)\equiv\delta\bmod{p^r},p^k\mmid u)$ directly follows from Lemma \ref{lem1}.  
Thus we can get the desired result.
\end{proof}

Combining Proposition \ref{count} and \ref{proppr}, 
we can get the explicit values of $\eta(D(p)\equiv\delta\bmod{p^r})$ as follows.

\begin{cor}\label{cor}
Let $p$ be a prime, $r\geq1$ an integer and $\delta\in\bZ_{p^r}$. 
Then $\eta(D(p)\equiv\delta\bmod{p^r})$ is described as follows.  

\noi{\bf The case of $p=2$.} 
\begin{align*}
&\eta(D(2)\equiv \delta\bmod{2^{r+2}})=\frac{1}{7\times2^{r+4}}\\
&\times\begin{cases} \disp 1,&(\delta\equiv 1\bmod{8}),\\
 75,&(\delta\equiv 5\bmod{8}),\\
 4+7\times2^{l+4},&(\delta\equiv -4+\alpha^22^{2l}\bmod{2^{r+2}}, 2\leq l<r/2),\\
 4+ 7\times 2^{[r/2]+3},&(\delta\equiv -4\bmod{2^{r+2}},
  \text{or $2\nmid r$, $\delta\equiv -4+2^{r+1}\bmod{2^{r+2}}$}),\\
 4,&(4\mmid\delta,\text{otherwise}),\\
 2^5,&(\text{$8\mmid \delta$ or $16\mmid \delta$}),\\
 2^8,&(32\mid \delta), \\
 0,&(\text{otherwise}).
\end{cases}
\end{align*} 

\noi{\bf The case of $p\geq3$.}
\begin{align*}
&\eta(D(p)\equiv \delta\bmod{p^{r}})=
\frac{1}{p^{r-1}(p^2-1)(p^3-1)}\\ \times
&\begin{cases} 
\disp \left(2+p^{[r/2]}(p^3-1)\right)\left(p+\left(\frac{\delta}{p}\right)\right),
&\disp \left(p\nmid \delta, \delta\equiv -4\bmod{p^r}\right),\\
\disp 2\left(1+p^{l}(p^3-1)\right)\left(p+\left(\frac{\delta}{p}\right)\right),
&\disp \big(p\nmid \delta, \delta\equiv -4+\alpha^2p^{2l}\bmod{p^r},\\ & 0\leq l<r/2\big),\\
\disp 2\left(p+\left(\frac{\delta}{p}\right)\right),
&\disp \left(p\nmid \delta, \text{otherwise}\right),\\
\disp 2p^2(p^2-1),&(p\mid \delta).
\end{cases}
\end{align*}\qed
\end{cor}

We further prepare the following lemma. 
This will be used to estimate the reminder terms in the main theorems. 

\begin{lem}\label{ubigT}
Let $x,T>0$ be large numbers with $T\ll x$. 
Then we have
\begin{align*}
\pi(x;\text{$m\mid u$ for $m>T$})=&O(x^2T^{-2+\epsilon}),\\
\pi(x;\text{$m^2\mid D$ for $m>T$})=&O(x^2T^{-1+\epsilon}) 
\end{align*}
for any $\epsilon>0$, where the implied constants depend on $\epsilon$.
\end{lem}

\begin{proof} 
It is known (see, e.g. \cite{Mo}) that 
there exists a constant $C>0$ such that 
\begin{align*}
h(D)\leq C\frac{\sqrt{D}\log{D}}{\log{\epsilon(D)}}.
\end{align*}
Then, according to Fact \ref{factquad}, we get 
\begin{align*}
&\pi(x;\text{$m\mid u$ for $m>T$})
<C\sum_{m>T}\sum_{\begin{subarray}{c}u\geq m\\ m\mid u\end{subarray}}
\sum_{\begin{subarray}{c}3\leq t\leq x \\ u^2\mid t^2-4\end{subarray}}
\sqrt{\frac{t^2-4}{u^2}}\frac{\log{\disp\frac{t^2-4}{u^2}}}{\log{t}}
<C\sum_{m>T}\sum_{\begin{subarray}{c}u\geq m\\ m\mid u\end{subarray}}u^{-1}
\sum_{\begin{subarray}{c}3\leq t\leq x \\ u^2\mid t^2-4\end{subarray}}t. 
\end{align*}
Since $u^2\mid t^2-4$ is equivalent that $t\equiv \pm2\bmod{p^{2k}}$ for any $p^k\mmid u$, 
we have
\begin{align*}
\pi(x;\text{$m\mid u$ for $m>T$})<
C\sum_{m>T}\sum_{\begin{subarray}{c}u\geq m\\ m\mid u\end{subarray}}\frac{2^{\omega(u)}}{u^3}x^2,
\end{align*}
where $\omega(u)$ is the number of distinct prime factors of $u$. 
Thus the estimate 
\begin{align}
\sum_{u<Y}2^{\omega(u)}=O(Y^{\epsilon}),\label{omega}
\end{align}
due to \cite{HarWr}, gives  
\begin{align*}
\pi(x;\text{$m\mid u$ for $m>T$})<&
C\sum_{m>T}\sum_{u}\frac{2^{\omega(u)}}{m^3u^3}x^2
<C\sum_{m>T}\frac{2^{\omega(m)}}{m^{3-\epsilon}}x^2=O(x^2T^{-2+\epsilon}).
\end{align*}
This completes the first estimation in the lemma. 

Similarly, we see that 
\begin{align*}
\pi(x;\text{$m^2\mid D$ for $m>V$})
<&\sum_{m>V}\sum_{u<x/m}\frac{1}{u}\sum_{\begin{subarray}{c}t<x\\ t^2-4\mid m^2u^2\end{subarray}}t
<\sum_{m>V}\sum_{u<x/m}\frac{2^{\omega(mu)}}{u^3m^2}x^2.
\end{align*}
Then we can get the second estimate due to \eqref{omega}.
\end{proof}

\subsection{Main theorem}

The main result in this paper is as follows. 
\begin{thm}\label{thm1}
Let $n\geq1$ be an integer with the factorization $n=\prod_{p\mid n}p^r$ and $\delta\in\bZ_n$. 
For $\alpha\in\bZ_n^*$, denote by 
\begin{align*}
W(\alpha;n):=\xi(n)\sum_{m\geq1,m\equiv\alpha\bmod{n}}\beta(m),
\end{align*}
where 
\begin{align*}
\xi(n):=&\prod_{p\nmid n}\left(1-v(p)^{-1}\right)
=\prod_{p\nmid n,p\neq 2}\left(1-\frac{2}{p(p^2-1)}\right)
\times \begin{cases}\disp \frac{5}{6},&(2\nmid 2),\\ 1,&(n\mid 2),\end{cases}\\
\beta(m):=&\prod_{p\mid m}\frac{v(p^l)^{-1}-v(p^{l+1})^{-1}}{1-v(p)^{-1}} 
=\prod_{p\mid m,p\neq 2}2p^{-3l}\frac{p^3-1}{p^3-p-2}\times 
\begin{cases}1,&(e=0),\\ \disp \frac{3}{20},&(e=1),\\ \disp \frac{3}{80},&(e=2),\\ \disp \frac{7}{5\times 2^{3e-2}},&(e\geq3),\end{cases}
\end{align*}
for $m=2^e\prod_{p\mid m,p\neq 2}p^l$. 
Then $\eta(D\equiv \delta\bmod{n})$ exists and satisfies that 
\begin{align*}
\eta(D\equiv \delta\bmod{n})
=\sum_{\alpha\in\bZ_n^{*}}W(\alpha;n)\prod_{p\mid n}\eta(D(p)\equiv\delta\alpha^2\bmod{p^{r}}).
\end{align*}
The reminder term is estimated by
\begin{align*}
R(x;D\equiv\delta\bmod{n})=O\left(x^{5/3+\epsilon}\right)
\end{align*}
for any $\epsilon>0$, where the implied constant depends on $n$ and $\epsilon$.
\end{thm}

\begin{rem}\label{Product}
It is easy to see that 
\begin{align*}
\sum_{\alpha\in \bZ_n^*}W(\alpha;n)=1.
\end{align*}
Then it holds that
\begin{align}
\eta\left(D\equiv \delta\bmod{n}\right)=\prod_{p\mid n}\eta\left(D(p)\equiv \delta\bmod{p^r}\right).
\end{align}
if the value
\begin{align*}
\prod_{p\mid n}\eta(D(p)\equiv\delta\alpha^2\bmod{p^{r}})
\end{align*}
does not depend on $\alpha\in\bZ_n^*$. 
For example, we have 
\begin{align}\label{mid}
\eta(n\mid D)=\prod_{p\mid n}\eta\left(p^r\mid D(p)\right).
\end{align}
\end{rem}

\begin{rem}
Let $\{\chi\}$ be the set of Dirichlet's characters modulo $n$ and $\varphi(n):=\#\bZ_n^*$. 
Then $W(\alpha;n)$ is written by 
\begin{align*}
W(\alpha;n):=\frac{1}{\varphi(n)}\sum_{\chi\bmod{n}}\chi(\alpha^{-1})\xi(n)\sum_{m\geq1,\gcd(n,m)=1}\chi(m)\beta(m).
\end{align*}
Since $\chi(n_1n_2)=\chi(n_1)\chi(n_2)$ and $\beta(n_1n_2)=\beta(n_1)\beta(n_2)$ for relatively prime $n_1,n_2$, 
we get
\begin{align*}
\sum_{m\geq1,\gcd(n,m)=1}\chi(m)\beta(m)=\prod_{p\nmid n}\sum_{l\geq0}\chi(p^l)\beta(p^l).
\end{align*}
Calculating the above carefully, we can obtain the following expression of $W(\alpha;n)$ like the Euler product.
\begin{align}
W(\alpha;n)=&\frac{1}{\varphi(n)}\sum_{\chi\bmod{n}}\chi(\alpha^{-1})
\prod_{p\geq3,p\nmid n}\left(1+\frac{2p^2(\chi(p)-1)}{(p^2-1)(p^3-\chi(p))}\right)\notag\\
&\times \begin{cases}1,&(2\mid n),\\ 
\disp \left(1+\frac{(\chi(2)-1)(32+4\chi(2)+\chi(2)^2)}{24(8-\chi(2))}\right),&(2\nmid n).\end{cases}
\end{align}
\end{rem}

\begin{proof}
Let $T>0$ be a number with $T\ll x$.  
Then, according to Lemma \ref{ubigT}, we have
\begin{align}\label{bigT}
\pi(x;D\equiv \delta\bmod{n})=\pi(x;D\equiv \delta\bmod{n},u<T)+O(x^2T^{-2+\epsilon}).
\end{align}
Since $D(n)=u_n^2D$ for given $n$, 
the first term in the right hand side is described as follows.
\begin{align*}
\pi(x;D\equiv\delta\bmod{n},u<T)
=&\sum_{m<T}\pi(x;D\equiv \delta\bmod{n},u=m)\\
=&\sum_{\alpha\in\bZ_n^*}\sum_{\begin{subarray}{c}m_n\equiv\alpha\bmod{n}\\ m<T\end{subarray}}
\pi\left( x;D(n)\equiv \delta\alpha^2\bmod{n},u=m \right).
\end{align*}
Here, the condition ``$u=m$" is equivalent to ``$m\mid u$ and $mm_1\nmid u$ for any $m_1>1$". 
Then we have
\begin{align*}
&\pi(x;D\equiv\delta\bmod{n},u<T)\\
=&\sum_{\alpha\in\bZ_n^*}\sum_{\begin{subarray}{c}m_n\equiv\alpha\bmod{n}\\ m<T\end{subarray}}
\sum_{m_1<T/m}\mu(m_1)\pi(x;D(n)\equiv\delta\alpha^2\bmod{n},mm_1\mid u)\\
=&\sum_{\alpha\in\bZ_n^*}\sum_{\begin{subarray}{c}m_n\equiv\alpha\bmod{n}\\ m<T\end{subarray}}
\sum_{m_1<T/m}\mu(m_1)\Big(\eta(D(n)\equiv\delta\alpha^2\bmod{n},mm_1\mid u)\li(x^2)\\
&+R(x;D_{m}\equiv\delta\alpha^2\bmod{n},mm_1\mid u)\Big)\\
=:&S_1(T)\li(x^2)+S_2(T,x),
\end{align*}
where $\mu(m)$ is the M\"{o}bius function. 
Due to Lemma \ref{lem1}, \ref{lempr} and \ref{factor}, we get
\begin{align}
&\left|S_2(U,x)\right|\notag\\
\leq &C\sum_{\alpha\in\bZ_n^*}\sum_{\begin{subarray}{c}m_n\equiv\alpha\bmod{n}\\ m<T\end{subarray}}
\sum_{m_1<T/m}\eta(D_m\equiv \delta\alpha^2\bmod{n},mm_1\mid u)v(nmm_1)x^{3/2}\notag\\
\leq &C\sum_{\alpha\in\bZ_n^*}\sum_{\begin{subarray}{c}m_n\equiv\alpha\bmod{n}\\ m<T\end{subarray}}
\sum_{m_1<T/m}n^3x^{3/2}\leq C'n^3x^{3/2}T\log{T}\label{errordelta}
\end{align} 
for some $C>0$.
On the other hand, since 
$\eta\left(D(n)\equiv\delta\alpha^2\bmod{n},mm_1\mid u\right)\leq \eta(mm_1\mid u)=v(mm_1)^{-1}\leq C(mm_1)^{-3}$, 
we have 
\begin{align}
S_1(T)=\lim_{T\tinf}S_1(T)+O(T^{-2}).\label{s1}
\end{align}
Thus, combining \eqref{bigT}, \eqref{errordelta} and \eqref{s1}, we get
\begin{align*}
&\eta(D\equiv\delta\bmod{n})\\&=
\sum_{\alpha\in\bZ_n^*}\sum_{\begin{subarray}{c}m_n\equiv\alpha\bmod{n}\\ m<T\end{subarray}}
\sum_{m_1<T/m}\mu(m_1)\eta(D(n)\equiv\delta\alpha^2\bmod{n},mm_1\mid u),\\
&R(x;D\equiv\delta\bmod{n})=O(x^{3/2}T\log{T})+O(T^{-2+\epsilon}x^2).
\end{align*}
Taking $T=x^{1/6}$, we can estimate the reminder term as in the theorem.  

We now study $\eta(D\equiv\delta\bmod{n})$. 
Denote by $n=p_1^{r_1}\cdots p_l^{r_l}$ the factorization of $n$. 
First calculate the sum over $m_1$.
\begin{align*}
&\sum_{m_1}\mu(m_1)\eta(D(n)\equiv\delta\alpha^2\bmod{n},mm_1\mid u)\\
=&\sum_{m_1:\gcd(m_1,n)=1}\mu(m_1)\sum_{0\leq k_1',\cdots,k_l'\leq1}(-1)^{k_1'+\cdots+k'_l}
\eta\left(D(n)\equiv\delta\alpha^2\bmod{n},p_1^{k'_1}\cdots p_l^{k'_l}mm_1\mid u\right)\\
=&\sum_{m_1:\gcd(m_1,n)=1}\mu(m_1)
\eta\left(D(n)\equiv\delta\alpha^2\bmod{n},mm_1\mid u,p_imm_1\nmid u (\forall i)\right).
\end{align*}
Next, take the sum over $m$. 
Due to Lemma \ref{factor}, we have
\begin{align*}
&\sum_{m:m_n\equiv\alpha\bmod{n}}\sum_{m_1:\gcd(m_1,n)=1}
\mu(m_1)\eta\left(D(n)\equiv\delta\alpha^2\bmod{n},mm_1\mid u,p_imm_1\nmid u (\forall i)\right)\\
=&\sum_{\begin{subarray}{c}m:\gcd(m,n)=1\\ m\equiv\alpha\bmod{n}\end{subarray}}
\sum_{k_1,\cdots,k_l\geq0}\sum_{m_1:\gcd(m_1,n)=1}\mu(m_1)\\
&\times \eta\left(D(n)\equiv\delta\alpha^2\bmod{n},mm_1\mid u,p_i^{k_i}\mmid u (\forall i)\right)\\
=&\sum_{\begin{subarray}{c}m:\gcd(m,n)=1\\ m\equiv\alpha\bmod{n}\end{subarray}}
\sum_{m_1:\gcd(m_1,n)=1}\mu(m_1)\eta(mm_1\mid u)\prod_{i=1}^{l}\left(\sum_{k_i\geq0}
\eta\left(D({p_i})\equiv n_{p_i}^2\delta\alpha^2\bmod{p_i^{r_i}},p_i^{k_i}\mmid u\right)\right)\\
=&\prod_{i=1}^{l}\eta\left(D({p_i})\equiv n_{p_i}^2\delta\alpha^2\bmod{p_i^{r_i}}\right)
\sum_{\begin{subarray}{c}m:\gcd(m,n)=1\\ m\equiv\alpha\bmod{n}\end{subarray}}
\sum_{m_1:\gcd(m_1,n)=1}\mu(m_1)\eta(mm_1\mid u).
\end{align*} 
Let $m=q_1^{i_1}\cdots q_t^{i_t}$ be the factorization of $m$. 
Since $\eta(ab\mid u)=\eta(a\mid u)\eta(b\mid u)$ for relatively prime $a,b$, 
we have
\begin{align*}
&\sum_{m_1:\gcd(m_1,n)=1}\mu(m_1)\eta(mm_1\mid u)\\
=&\sum_{m_1:\gcd(m_1,nm)=1}\sum_{0\leq j_1,\cdots,j_t\leq1}\mu(q_1^{j_1}\cdots q_t^{j_t}m_1)
\eta(q_1^{i_1+j_1}\cdots q_t^{i_t+j_t}m_1\mid u)\\
=&\sum_{m_1:\gcd(m_1,nm)=1}\mu(m_1)\eta(m_1\mid u)
\prod_{k=1}^t\left(\eta(q_k^{i_k}\mid u)-\eta(q_k^{i_k+1}\mid u)\right)\\
=&\prod_{p\nmid nm}\left(1-\eta(p\mid u)\right)
\prod_{k=1}^{t}\left(\eta(q_k^{i_k}\mid u)-\eta(q_k^{i_k+1}\mid u)\right).
\end{align*}
Since $\eta(n\mid u)=v(n)^{-1}$, the above coincides $\xi(n)\beta(m)$ 
and then the expression of $\eta(D\equiv\delta\bmod{n})$ in the theorem follows immediately.
\end{proof}

\vbpt 

\noi{\bf Example.} As an example, we study the case of $n=5$. 
According to Theorem \ref{thm1}, we have
\begin{align*}
\eta(D\equiv 0)=&\eta(D(5)\equiv 0),\\
\eta(D\equiv 1)=&\eta(D(5)\equiv 1)W_1+\eta(D(5)\equiv 4)W_2,\\
\eta(D\equiv 2)=&\eta(D(5)\equiv 2)W_1+\eta(D(5)\equiv 3)W_2,\\
\eta(D\equiv 3)=&\eta(D(5)\equiv 3)W_1+\eta(D(5)\equiv 2)W_2,\\
\eta(D\equiv 4)=&\eta(D(5)\equiv 4)W_1+\eta(D(5)\equiv 1)W_2,
\end{align*} 
where $W_1:=W(1;5)+W(4;5)$ and $W(2):=W(2;5)+W(3;5)$. 
Due to Corollary \ref{cor}, we see that 
\begin{align*}
\eta(D(5)\equiv 0)=\frac{25}{62},\quad
&\eta(D(5)\equiv 1)=\frac{63}{248},
\qquad \eta(D(5)\equiv 2)=\frac{125}{372},\\
&\eta(D(5)\equiv 3)=\frac{1}{372},
\qquad \eta(D(5)\equiv 4)=\frac{1}{248}.
\end{align*}
We can compute that $W_1=0.80233\cdots$ and $W_2=0.19766\cdots$. 
Thus we get the following approximations.
\begin{align*}
\eta(D\equiv 0)=&0.40322\cdots,\\
\eta(D\equiv 1)=&0.20461\cdots,\qquad
\eta(D\equiv 2)=0.27013\cdots,\\
\eta(D\equiv 3)=&0.06857\cdots,\qquad
\eta(D\equiv 4)=0.05344\cdots.
\end{align*}

\subsection{For fundamental discriminants} 

Let $d$ be $D$ when the square free factor of $D$ is $1$ modulo $4$ 
and $D/4$ otherwise. 
Denote by $\Nsf$ the set of positive square free integers. 
Raulf \cite{Ra} proved that 
\begin{align*}
\eta(d\in \Nsf)=&\frac{75}{112}\prod_{p\geq3}\left(1-\frac{2p}{p^3-1}\right)=0.42699\cdots,\\
R(x;d\in \Nsf)=&O(x^{2-\epsilon})
\end{align*}
for some $\epsilon>0$. 
We now extend it as follows.

\begin{thm}\label{thm2}
Let $n\geq1$ be an integer factored by $n=2^e\prod_{p\mid n}p^r$ and $\delta\in \bZ_n$.
Then $\eta(d\in\Nsf, d\equiv \delta\bmod{n})$ exists and satisfies that 
\begin{align*}
&\eta(d\in \Nsf,d\equiv \delta\bmod{n})\\
=&\prod_{p\nmid 2n}\left(1-\frac{2p}{p^3-1}\right)
\sum_{\alpha\in\bZ_{n'}^*}W(\alpha;n')\\
&\times\Bigg[ \eta(D(2)\equiv 1\bmod{4},D(2)\equiv \delta\alpha^2\bmod{2^e})
\prod_{p\mid n} \eta(D(p)\equiv \delta\alpha^2\bmod{p^{r}},p^2\nmid D)\\
&+\eta(D(2)\equiv 8,12\bmod{16},D(2)\equiv 4\delta\alpha^2\bmod{2^{e+2}})
\prod_{p\mid n} \eta(D(p)\equiv 4\delta\alpha^2\bmod{p^{r}},p^2\nmid D)\Bigg],
\end{align*}
where $n':=\lcm(16,4n,\prod_{p\mid n}p^2)$.
The reminder term is estimated by 
\begin{align*}
R(d\in\Nsf, d\equiv \delta\bmod{n})=O(x^{25/13+\epsilon}).
\end{align*}
\end{thm}

\begin{proof}
Due to Lemma \ref{ubigT}, we have
\begin{align}\label{uT}
&\pi(x;d\in\Nsf,d\equiv\delta\bmod{n})\notag\\
&=\pi(x;d\in\Nsf,d\equiv\delta\bmod{n},u<T)+O(T^{-2+\epsilon}x^2).
\end{align}
By the definition of $d$, we see that 
\begin{align*}
&\pi(x;d\in\Nsf,d\equiv\delta\bmod{n},u<T)\\
=&\sum_{\begin{subarray}{c}1\leq m\leq x\\ \text{odd}\end{subarray}}
\mu(m)\Big(\pi(x;D\equiv\delta\bmod{n},D\equiv 1\bmod{4},m^2\mid D,u<T)\\
&+\pi(x;D\equiv4\delta\bmod{4n},D\equiv 8,12\bmod{16},m^2\mid D,u<T)  \Big).
\end{align*}
Divide the sum above by 
\begin{align*}
\sum_{m<x}=\sum_{m<U}+\sum_{m\geq U}=:M_1(x,T,U)+M_2(x,T,U),
\end{align*}
where $U>0$ is a large number with $U\ll x$. 
According to Lemma \ref{ubigT}, we have 
\begin{align}
M_2(x,T,U)=O(U^{-1+\epsilon}x^2).\label{M2}
\end{align}
We further divide $M_1(x,T,U)$ by 
\begin{align*}
M_1(x,T,U)&=\sum \pi(x;\cC)=\sum \eta(\cC)\li(x^2)+\sum R(x;\cC)\\
&=:M_{11}(T,U)\li(x^2)+M_{12}(x;T,U).
\end{align*}
According to Corollary \ref{cor} and Theorem \ref{thm1}, 
we see that 
\begin{align}
\eta(D\equiv\delta\bmod{n},D\equiv 1\bmod{4},m^2\mid D, u<U)\leq 
\eta(m^2\mid D)\leq Cm^{-2}\label{mmidD}
\end{align}
for some $C>0$ and $\eta(D\equiv4\delta\bmod{4n},D\equiv 8,12\bmod{16},m^2\mid D, u<U)$ 
is similarly. Then we have 
\begin{align}
M_{11}(T,U)=\lim_{T,U\tinf}M_{11}(T,U,x)+O(U^{-1}).\label{M11}
\end{align}
On the other hand, by taking discussions similar to $S_2(T,x)$ 
in the proof of Theorem \ref{thm1}, we can estimate  
\begin{align}
|M_{2}(T,U,x)|\leq &C\sum_{m<U} n^3m^4x^{3/2}T\log{T}
=O\left( U^5T^{1+\epsilon}x^{3/2} \right).\label{M12}
\end{align}
Combining \eqref{uT}, \eqref{M2}, \eqref{M11} and \eqref{M12}, we get 
\begin{align*}
&\eta(d\in \Nsf,d\equiv \delta\bmod{n})\\&=
\sum_{\begin{subarray}{c}m\geq1\\ \text{odd} \end{subarray}}\mu(m)
\bigg(\eta(D\equiv\delta\bmod{n},D\equiv 1\bmod{4},m^2\mid D)\\
&+\eta(D\equiv4\delta\bmod{4n},D\equiv 8,12\bmod{16},m^2\mid D)\bigg),\\
&R(x;d\in \Nsf,d\equiv \delta\bmod{n})=
O(U^{-1+\epsilon}x^2)+O(T^{-2+\epsilon}x^2)+O(U^5T^{1+\epsilon}x^{3/2}).
\end{align*}
Thus the estimate of the reminder term in the theorem
follows immediately from the above with $T=x^{1/26}$ and $U=x^{1/13}$. 

We now arrange the expressions of $\eta(d\in \Nsf,d\equiv \delta\bmod{n})$. 
Change the sum over $m$ by 
\begin{align*}
&\eta(d\in\Nsf,d\equiv\delta\bmod{n})\\
=&\sum_{m:\gcd(m,2n)=1}\sum_{0\leq j_1,\cdots,j_l\leq1}\mu(m)(-1)^{j_1+\cdots+j_l}\\
&\times\bigg(\eta(D\equiv\delta\bmod{n},D\equiv 1\bmod{4},m^2\mid D,p_i^{2j_i}\mid D(\forall i))\\
&+\eta(D\equiv4\delta\bmod{4n},D\equiv 8,12\bmod{16},m^2\mid D,p_i^{2j_i}\mid D(\forall i))\bigg)\\
=&\sum_{m:\gcd(m,2n)=1}\mu(m)
\bigg(\eta(D\equiv\delta\bmod{n},D\equiv 1\bmod{4},m^2\mid D,p_i^{2}\nmid D(\forall i))\\
&+\eta(D\equiv4\delta\bmod{4n},D\equiv 8,12\bmod{16},m^2\mid D,p_i^{2}\nmid D(\forall i))\bigg).
\end{align*}
According to Theorem \ref{thm1}, we have 
\begin{align*}
&\eta(D\equiv\delta\bmod{n},D\equiv 1\bmod{4},m^2\mid D,p_i^{2}\nmid D(\forall i))\\
=&\sum_{\alpha\in\bZ_{n'm^2}^*}W(\alpha;n'm^2)\eta(m^2\mid D)
\eta(D(2)\equiv 1\bmod{4},D(2)\equiv \delta\alpha^2\bmod{2^e})\\
&\times \prod_{i=1}^{l} \eta(D(p_i)\equiv \delta\alpha^2\bmod{p_i^{r_i}},p_i^2\nmid D).
\end{align*}
Similarly, we have
\begin{align*}
&\eta\left(D\equiv4\delta\bmod{4n},D\equiv 8,12\bmod{16},m^2\mid D,p_i^{2}\nmid D(\forall i)\right)\\
=&\sum_{\alpha\in\bZ_{n'm^2}^*}W(\alpha;n'm^2)\eta(m^2\mid D)
\eta(D(2)\equiv 8,12\bmod{16},D(2)\equiv 4\delta\alpha^2\bmod{2^{e+2}})\\
&\times \prod_{i=1}^{l} \eta(D(p_i)\equiv 4\delta\alpha^2\bmod{p_i^{r_i}},p_i^2\nmid D).
\end{align*}
Thus the following formula holds.
\begin{align*}
&\eta(d\in \Nsf,d\equiv \delta\bmod{n})=\sum_{\begin{subarray}{c}m\geq1,\gcd(m,2n)=1\end{subarray}}\mu(m)\eta(m^2\mid D)
\sum_{\alpha\in\bZ_{n'm^2}^*}W(\alpha;n'm^2)\sigma(\alpha;n'),
\end{align*}
where
\begin{align*}
&\sigma(\alpha;n')\\:=& \eta(D(2)\equiv 1\bmod{4},D(2)\equiv \delta\alpha^2\bmod{2^e})
\prod_{p\mid n} \eta(D(p)\equiv \delta\alpha^2\bmod{p^{r}},p^2\nmid D)\\
&+\eta(D(2)\equiv 8,12\bmod{16},D(2)\equiv 4\delta\alpha^2\bmod{2^{e+2}})
\prod_{p\mid n} \eta(D(p)\equiv 4\delta\alpha^2\bmod{p^{r}},p^2\nmid D).
\end{align*}
By the definition of $W(*;*)$, we get 
\begin{align*}
&\eta(d\in \Nsf,d\equiv \delta\bmod{n})
=\sum_{\begin{subarray}{c}m\geq1,\gcd(m,2n)=1\end{subarray}}\mu(m)\eta(m^2\mid D)
\sum_{\alpha\in\bZ_{n'}^*}W(\alpha;n')\sigma(\alpha;n').
\end{align*}
According to Remark \ref{Product}, we have
\begin{align*}
\eta(m^2\mid D)=\prod_{p\mid m}\eta(p^i\mid D(p))=\prod_{p\mid m}\frac{2p^{3-i}}{p^3-1}
\end{align*}
for $2\nmid m$ and $m=\prod_{p\mid m}p^{i}$. 
Thus the desired result follows immediately.
\end{proof}

\noi{\bf Example.} Study the case of $n=5$. 
According to Theorem \ref{thm2}, we have
\begin{align*}
\eta(d\in \Nsf,d\equiv 0\bmod{5})=&\Omega\eta(5\mmid D(5))\big\{\eta(D(2)\equiv 1\bmod{4})\\&+\eta(D(2)\equiv 8,12\bmod{16})\big\},\\
\eta(d\in \Nsf,d\equiv 1\bmod{5})=&\Omega\left[V_1\eta(D(5)\equiv 1\bmod{5})+V_2\eta(D(5)\equiv 4\bmod{5}) \right],\\
\eta(d\in \Nsf,d\equiv 2\bmod{5})=&\Omega\left[V_1\eta(D(5)\equiv 2\bmod{5})+V_2\eta(D(5)\equiv 3\bmod{5}) \right],\\
\eta(d\in \Nsf,d\equiv 3\bmod{5})=&\Omega\left[V_1\eta(D(5)\equiv 3\bmod{5})+V_2\eta(D(5)\equiv 2\bmod{5}) \right],\\
\eta(d\in \Nsf,d\equiv 4\bmod{5})=&\Omega\left[V_1\eta(D(5)\equiv 4\bmod{5})+V_2\eta(D(5)\equiv 1\bmod{5}) \right],
\end{align*}
where 
\begin{align*}
\Omega:=&\prod_{p\neq 2,5}\left(1-\frac{2p}{p^3-1}\right)=0.69357\cdots,\\
V_1:=&W_1\eta(D(2)\equiv 1\bmod{4})+W_2\eta(D(2)\equiv 8,12\bmod{16}),\\
V_2:=&W_1\eta(D(2)\equiv 8,12\bmod{16})+W_2\eta(D(2)\equiv 1\bmod{4}).
\end{align*}

Since $W_1,W_2$ and $\eta(D(5)\equiv \delta\bmod{5})$ are given in the example
for Theorem \ref{thm1} and 
\begin{align*}
&\eta(D(2)\equiv 1\bmod{4})=\frac{19}{56},\quad \eta(D(2)\equiv 8,12\bmod{16})=\frac{37}{112},\\
&\eta(5\mmid D(5))=\eta(5\mid D(5))-\eta(25\mid D(5))=\frac{10}{31},
\end{align*}
we get the following approximations. 
\begin{align*}
\eta(d\in \Nsf,d\equiv 0)=&0.1498\cdots,\\
\eta(d\in \Nsf,d\equiv 1)=&0.0603\cdots,\quad 
\eta(d\in \Nsf,d\equiv 2)=0.0792\cdots,\\
\eta(d\in \Nsf,d\equiv 3)=&0.0780\cdots,\quad
\eta(d\in \Nsf,d\equiv 4)=0.0594\cdots.
\end{align*}

\noindent 
HASHIMOTO, Yasufumi  \\
Institute of Systems and Information Technologies/KYUSHU,\\
7F 2-1-22, Momochihama, Fukuoka 814-0001, JAPAN\\
e-mail:hasimoto@isit.or.jp

\end{document}